\documentclass{article}
\usepackage{amsmath,amssymb,amscd,amsthm,verbatim,alltt,amsfonts,array}
\usepackage[english]{babel}
\usepackage{latexsym}
\usepackage{amssymb}
\usepackage{euscript}
\usepackage{graphicx}
\usepackage{epsfig}

\newtheorem{theorem}{Theorem}

\newtheorem{definition}{Definition}
\newtheorem{example}{Example}

\newtheorem{lemma}{Lemma}

\newtheorem{remark}{Remark}

\title{On  the solutions of a second-order  difference equations  in
terms of generalized Padovan sequences}

\author{ Yacine Halim$^{1}$ and Julius Fergy T. Rabago$^{2}$   \\
 $^{1}$ Department of Mathematics and computer sceince\\
Mila University Center, Mila, Algeria\\
Email: halyacine@yahoo.fr\\
 $^{2}$ Department of Mathematics and computer sciences,\\
   College of Science, University of the Philippines,\\
Gov. Pack Road, Baguio City 2600, Benguet, Philippines.\\
Email: jfrabago@gmail.com\\
}
\date{}

\begin{document}
\maketitle

\begin{abstract}
This paper deals with the solution, stability character and asymptotic behavior
of the rational difference equation
\begin{equation*}
x_{n+1}=\frac{\alpha x_{n-1}+\beta}{ \gamma x_{n}x_{n-1}},\qquad n
\in \mathbb{N}_{0},
\end{equation*}
where  $\mathbb{N}_{0}=\mathbb{N}\cup \left\{0\right\}$,
$\alpha,\beta,\gamma\in\mathbb{R}^{+}$, and the initial
conditions $x_{-1}$ and $x_{0}$ are non zero real numbers such
that their solutions are associated to generalized Padovan numbers.
Also, we investigate the two-dimensional case of the this equation given by
\begin{equation*}
	x_{n+1} = \frac{\alpha x_{n-1}  + \beta}{\gamma y_n x_{n-1}}, \qquad y_{n+1} = \frac{\alpha y_{n-1} +\beta}{\gamma x_n y_{n-1}} ,\qquad n\in \mathbb{N}_0,
\end{equation*}
and this generalizes the results presented in \cite{yazlik}.
\\
 \textbf{Keywords}: Difference equations,
general solution,  stability, generalized Padovan numbers.\\
\textbf{Mathematics Subject Classification:} 39A10, 40A05.

\end{abstract}
\section{Introduction and  preliminaries }
The term difference equation refers to a specific type of recurrence relation -- a mathematical relationship expressing $x_n$ as some combination of $x_i$ with $i<n$.
These equations usually appear as discrete mathematical models of many biological and environmental phenomena such as population growth and predator-prey interactions (see, e.g., \cite{book} and \cite{book2}),
and so these equations are studied because of their rich and complex dynamics.
Recently, the problem of finding closed-form solutions of rational difference equations and systems of rational of difference equations have gained considerable interest from many mathematicians.
In fact, countless papers have been published previously focusing on the aforementioned topic, see for example \cite{elsayed, elsayed0, elsayed1, khaliq, rabago1} and \cite{rabago2}.
Interestingly, some of the solution forms of these equations are even expressible in terms of well-known integer sequences such as the Fibonacci numbers, Horadam numbers and Padovan numbers (see, e.g., \cite{hal1,hal6,hal2,hal4,stv1, tollu, tollu1, tollu2, touafek,T5, yazlik}).

It is well-known that linear recurrences with constant coefficients, such as the recurrence relation $F_{n+1} = F_n + F_{n-1} $ defining the Fibonacci numbers, can be solved through various techniques (see, e.g.,  \cite{larcombe}).
Finding the closed-form solutions of nonlinear types of difference equations, however, are far more interesting and challenging compared to those of linear types.
In fact, as far as we know, there has no known general method to deal with different classes of difference equations solvable in closed-forms.
Nevertheless, numerous studies have recently dealt with finding appropriate techniques in solving closed-form solutions of some systems of difference equations (see, e.g., \cite{ brand,elsayed, elsayed0, elsayed1, hal5,stv2}).

Motivated by these aforementioned works, we investigate the rational difference equation
\begin{equation}
x_{n+1}=\frac{\alpha x_{n-1}+\beta}{ \gamma x_{n}x_{n-1}}, \quad n
\in \mathbb{N}_{0},.\label{eq0}
\end{equation}
Particularly, we seek to find its closed-form solution and examine the global stability of its positive solutions.
We establish the solution form of equation \eqref{eq0} using appropriate transformation reducing the equation into a linear type difference equation.
Also, we examine the solution form of the two-dimensional analogue of equation \eqref{eq0} given in the following more general form
\begin{equation}
	x_{n+1} = \frac{\alpha x_{n-1}  + \beta}{\gamma y_n x_{n-1}}, \qquad y_{n+1} = \frac{\alpha y_{n-1} +\beta}{\gamma x_n y_{n-1}} ,\qquad n\in \mathbb{N}_0.\label{eq01}
\end{equation}
The case $ \alpha=\beta=\gamma=1$ has been studied by Tollu, Yazlik and Taskara in \cite{yazlik}.
The authors in \cite{yazlik} established the solution form of system \eqref{eq01} (in the case $ \alpha=\beta=\gamma=1$) through induction principle.

The paper is organized as follows.
In the next section (Section \ref{sec2}), we review some definitions and important results necessary for the success of our study,
and this includes a brief discussion about generalized Padovan numbers.
In section \ref{sec3} and \ref{sec4}, we established the respective solution forms of equations \eqref{eq0} and the system \eqref{eq01}, and examine their respective stability properties.
Finally, we end our paper with a short summary in Section \ref{sec5}.

\section{Preliminaries}\label{sec2}
\subsection{Linearized stability of an equation }
Let $I$ be an interval of real numbers and let
$$F:I^{k+1}\longrightarrow I$$ be a continuously differentiable
function. Consider the difference equation
\begin{equation}
x_{n+1}=F(x_{n},x_{n-1},\ldots,x_{n-k})\label{eq2}
\end{equation}
with initial values $x_{0},x_{-1},\ldots \text{ }x_{-k}\in I$..
\begin{definition}
A point $\overline{x}\in I$ is called an equilibrium point of
equation\eqref{eq2} if
$$\overline{x}=F(\overline{x},\overline{x},\ldots,\overline{x}).$$
\end{definition}
\begin{definition} Let $\overline{x}$ be an equilibrium point of
equation\eqref{eq2}.

\begin{itemize}

\item[i)]The equilibrium $\overline{x}$ is called locally stable if for every
$\varepsilon > 0$, there exist $\delta > 0$ such that for
all$x_{-k},x_{-k+1},\ldots \text{ }x_{0}\in I$ with
$$|x_{-k}-\overline{x}|+|x_{-k+1}-\overline{x}|+\ldots+|x_{0}-\overline{x}|< \delta,$$ we have
$|x_{n}-\overline{x}|<\varepsilon$, for all $n\geq -k$.

\item[ii)] The equilibrium $\overline{x}$ is called locally
asymptotically stable if it is locally stable, and if there exists
$\gamma > 0$ such that if $x_{-1}, \text{ }x_{0}\in I$ and
$$|x_{-k}-\overline{x}|+|x_{-k+1}-\overline{x}|+\ldots+|x_{0}-\overline{x}|<  \gamma,$$
 then
$$\lim_{n\rightarrow +\infty} x_{n}=\overline{x}.$$

\item[iii)] The equilibrium $\overline{x}$ is called global attractor if
for all $x_{-k},x_{-k+1},\ldots \text{ }x_{0}\in I$, we have
$$\lim_{n\rightarrow +\infty} x_{n}=\overline{x}.$$

\item[iv)] The equilibrium $\overline{x}$ is called global
asymptotically stable if it is locally stable and a global
attractor.

\item[v)] The equilibrium $\overline{x}$ is called  unstable if it is
not stable.

\item[vi)] Let $p_{i}=\frac{\partial f}{\partial u_{i}}(\overline{x},\overline{x},\ldots,\overline{x}),\ i=0,1,\ldots,k$. Then, the equation
\begin{equation}
y_{n+1}=p_{0}y_{n}+p_{1}y_{n-1}+\ldots+p_{k}y_{n-k},\label{eq03}
\end{equation}
 is called the linearized equation of equation \eqref{eq2} about the equilibrium point
$\overline{x}$.
\end{itemize}
\end{definition}
The next result, which was given by Clark \cite{clk}, provides a
sufficient condition for the locally asymptotically stability of equation
\eqref{eq2}.
\begin{theorem}[\cite{clk}] Consider the difference
equation \eqref{eq03}. Let $p_{i}\in\mathbb{R}$, then,
$$|p_{0}|+|p_{1}|+\ldots+|p_{k}|<1$$ is a sufficient condition for the
locally asymptotically stability of equation \eqref{eq2}.\label{th6}
\end{theorem}
\subsection{Linearized stability of the second-order systems }
Let $f$ and $g$ be two continuously differentiable functions:
    \[
        f:\;I^{2}\times J^{2}\longrightarrow I,\quad g:\;I^{2}\times J^{2}\longrightarrow J, \quad I,J \subseteq \mathbb{R}
    \]
and for $n \in \mathbb{N}_{0}$, consider the system of difference
equations
\begin{equation}
    \left\{
    \begin{array}{ll}
        x_{n+1}= & f\left( x_{n},x_{n-1},y_{n},y_{n-1}\right)  \\
        y_{n+1}= & g\left( x_{n},x_{n-1},y_{n},y_{n-1}\right)
    \end{array}
    \right.\label{eq111}
\end{equation}
where $\left(x_{-1},x_{0}\right)$ $\in I^{2}$ and
$\left(y_{-1},y_{0}\right)$ $\in J^{2}$.
Define the map
    $
    H:\;I^{2}\times J^{2}\longrightarrow I^{2}\times J^{2}
    $
 by
    \[
        H(W)=(f_{0}(W),f_{1}(W),g_{0}(W),g_{1}(W))
    \]
where
$W   =(u_{0},u_{1},v_{0},v_{1})^{T}$,
$f_{0}(W)    =f(W)$, $f_{1}(W)=u_{0}$,
$g_{0}(W)    =g(W)$, $g_{1}(W)=v_{0}$.
Let $W_{n}=\left[x_{n},x_{n-1},y_{n},y_{n-1}\right]^{T}$.
Then, we can easily see that system \eqref{eq111} is equivalent to
the following system written in vector form
    \begin{equation}
        W_{n+1}=H(W_{n}),\;n=0,1,\ldots,\label{eq112}
    \end{equation}
that is
\begin{equation*}
\left\{
\begin{array}{rcl}
    x_{n+1} &= & f\left( x_{n},x_{n-1},y_{n},y_{n-1}\right)  \\
    x_{n}       &= & x_{n} \\
    y_{n+1} &= & g\left( x_{n},x_{n-1},y_{n},y_{n-1}\right)  \\
    y_{n}       &= & y_{n}

\end{array}
\right..
\end{equation*}

\begin{definition}[Equilibrium point]\ An equilibrium point $(\overline{x},\overline{y})\in I\times J$ of system \eqref{eq111} is a solution of the system
    \[
        \left\{\begin{array}{l}
            x=f\left(x,x,y,y\right),\\
            y=g\left(x,x,y,y\right).
    \end{array}\right.
    \]
Furthermore, an equilibrium point $\overline{W}\in I^{2}\times
J^{2}$ of system \eqref{eq112} is a solution of the system
    \[
        W=H(W).
    \]
\end{definition}

\begin{definition}[Stability] Let $\overline{W}$ be an equilibrium point of system \eqref{eq112} and $\parallel .\parallel$ be any norm (e.g. the Euclidean norm).
\begin{enumerate}
  \item The equilibrium point $\overline{W}$ is called stable (or locally stable) if for every $\varepsilon >0$ exist $\delta$ such that $\| W_{0}-\overline{W} \|<\delta$ implies $\| W_{n}-\overline{W}
  \|<\varepsilon$ for $n\geq0$.

   \item The equilibrium point $\overline{W}$ is called asymptotically stable (or locally asymptotically stable) if it is stable and there exist $\gamma >0$ such that $\| W_{0}-\overline{W} \|<\gamma$ implies
    \[
        \| W_{n}-\overline{W} \|\rightarrow0,\,n\rightarrow+\infty.
    \]
  \item The equilibrium point $\overline{W}$ is said to be global attractor (respectively global attractor with basin of attraction a set  $G\subseteq I^{2}\times J^{2}$, if for every $W_{0}$ (respectively for every  $W_{0}\in G$)
    \[
        \| W_{n}-\overline{W} \|\rightarrow0,\,n\rightarrow+\infty.
    \]
  \item The equilibrium point $\overline{W}$ is called globally asymptotically stable (respectively globally asymptotically stable relative to $G$) if it is asymptotically stable, and if for every $W_{0}$ (respectively for every $W_{0}\in G$),
    \[
        \| W_{n}-\overline{W}\|\rightarrow0,\,n\rightarrow+\infty.
    \]
  \item The equilibrium point  $\overline{W}$ is called unstable if it is not stable.
  \end{enumerate}
\end{definition}

\begin{remark}
    Clearly, $(\overline{x},\overline{y})\in I\times J$ is an equilibrium point for system \eqref{eq111} if and only if
    $\overline{W}=(\overline{x},\overline{x},,\overline{y},\overline{y},)\in I^{2}\times J^{2}$ is an equilibrium point of system \eqref{eq112}.
\end{remark}

From here on, by the stability of the equilibrium points of system
\eqref{eq111}, we mean the stability of the corresponding equilibrium
points of the equivalent system \eqref{eq112}.

\subsection{Generalized Padovan sequence}
The integer sequence defined by the recurrence relation
\begin{equation}
\mathcal{P}_{n+1}=\mathcal{P}_{n-1}+\mathcal{P}_{n-2},\quad n\in
\mathbb{N},
\end{equation}
with the initial conditions $\mathcal{P}_{-2}=0$, $\mathcal{P}_{-1}=0$, $\mathcal{P}_{0}=1$ (so $\mathcal{P}_{0} = \mathcal{P}_{1} = \mathcal{P}_{2}=1$), is known as the Padovan numbers and was named after Richard Padovan.
This is the same recurrence relation as for the {\it Perrin sequence}, but with different initial conditions ($P_0 = 3, P_1 =0 , P_2 =2$).
The first few terms of the recurrence sequence are $1, 1, 2, 2, 3, 4, 5, 7, 9, 12, \ldots$.
The Binet's formula for this recurrence sequence can easily be
obtained and is given by
\[\mathcal{P}_{n}=\frac{(\rho-1)(\overline{\rho}-1)}{(\sigma-\rho)(\sigma-\overline{\rho})}\sigma^{n}
+\frac{(\sigma-1)(\overline{\rho}-1)}{(\rho-\sigma)(\rho-\overline{\rho})}\rho^{n}+
\frac{(\sigma-1)(\rho-1)}{(\sigma-\overline{\rho})(\rho-\overline{\rho})}\overline{\rho}^{n}.\]
where $\sigma=\frac{r^{2}+12}{6r}$ (the so-called 'plastic number),
$\rho=-\frac{\sigma}{2}+i\frac{\sqrt{3}}{2}\left(\frac{r}{6}-\frac{2}{r}\right)$
and $r=\sqrt[3]{108+12\sqrt{69}}$.
The plastic number corresponds to the golden number $\frac{1+\sqrt{5}}{2}$ associated with the
equiangular spiral related to the conjoined squares in Fibonacci
numbers, that is,
\[\lim_{n\rightarrow\infty}\frac{\mathcal{P}_{n+1}}{\mathcal{P}_{n}}=\sigma.\]
For more informations associated with Padovan sequence, see \cite{pad1} and \cite{pad2}.\\
\newline Here we define an extension of the Padovan sequence in the following way
\begin{equation}
\mathcal{S}_{-2}=0, \quad \mathcal{S}_{-1}=0, \quad \mathcal{S}_{0}=1, \quad \mathcal{S}_{n+1}=p\mathcal{S}_{n-1}+q\mathcal{S}_{n-2},\quad n\in
\mathbb{N}.
\end{equation}
The Binet's formula for this recurrence sequence  is given by
\[\mathcal{S}_{n}=\frac{(\varphi-1)(\overline{\varphi}-1)}{(\phi-\varphi)(\phi-\overline{\varphi})}\phi^{n}
+\frac{(\phi-1)(\overline{\varphi}-1)}{(\varphi-\phi)(\varphi-\overline{\varphi})}\varphi^{n}+
\frac{(\phi-1)(\varphi-1)}{(\phi-\overline{\varphi})(\varphi-\overline{\varphi})}\overline{\varphi}^{n}.\]
where $\phi=\frac{R^{2}+12p}{6R}$,
$\varphi=-\frac{\phi}{2}+i\frac{\sqrt{3}}{2}\left(\frac{R}{6}-\frac{2p}{R}\right)$
and $R=\sqrt[3]{108q+12\sqrt{-12p^{3}+81q^{2}}}$.
One can easily verify that \[\lim_{n\rightarrow\infty}\frac{\mathcal{S}_{n+1}}{\mathcal{S}_{n}}=\phi.\]

\section{Closed-Form  solutions and stability  of equation \eqref{eq0} }\label{sec3}
For the rest of our discussion we assume $\mathcal{S}_{n}$, the $n$-th generalized Padovan number, to
satisfy the recurrence equation
\[ \mathcal{S}_{n+1}=p\mathcal{S}_{n-1}+q\mathcal{S}_{n-2}, \quad n\in
\mathbb{N}_0,\]
with initial conditions
$\mathcal{S}_{-2}=0$, $\mathcal{S}_{-1}=0$, $\mathcal{S}_{0}=1$.

\subsection{Closed-Form  solutions of equation \eqref{eq0}}
In this section, we derive the solution form of equation \eqref{eq0}
 through an analytical approach. We put  $q=\dfrac{\alpha}{\gamma}$ and $p=\dfrac{\beta}{\gamma}$,
hence we have the equation
\begin{equation}
x_{n+1}=\frac{p x_{n-1}+q}{  x_{n}x_{n-1}};\quad n \in
\mathbb{N}_{0}.\label{eq1}
\end{equation}
Consider the equivalent form of
equation \eqref{eq1}  given by
\[x_{n+1}=\frac{p}{x_{n}}+\frac{q}{x_{n}x_{n-1}}\]
which, upon the change of variable $x_{n+1}=z_{n+1}/z_{n}$, transforms into
\begin{equation}
z_{n+1}=pz_{n-1}+qz_{n-2}.\label{eq4}
\end{equation}
Now, we iterate the right hand
side of equation \eqref{eq4} as follows
\begin{eqnarray*}
z_{n+1}&=&pz_{n-1}+qz_{n-2}\\
&=&qz_{n-2}+p^{2}z_{n-3}+qpz_{n-4}\\
&=&p^{2}z_{n-3}+2pqz_{n-4}+q^{2}z_{n-5}\\
&=&2pqz_{n-4}+(p^{3}+q^{2})z_{n-5}+qp^{2}z_{n-6}\\
&=&(p^{3}+q^{2})z_{n-5}+3p^{2}qz_{n-6}+2pq^{2}z_{n-7}\\
&=&3p^{2}qz_{n-6}+(p^{4}+3pq^{2})z_{n-7}+(p^{3}+q^{3})z_{n-8}\\
&=&(p^{4}+3pq^{2})z_{n-7}+(q^{3}+4qp^{3})z_{n-8}+3p^{2}q^{2}z_{n-9}\\
&\vdots&\\
&=&\mathcal{S}_{n+1}z_{0}+\mathcal{S}_{n+2}z_{
-1}+\mathcal{S}_{n}qz_{-2}.
\end{eqnarray*}
Hence,
\begin{eqnarray*}
x_{n+1}=\frac{z_{n+1}}{z_{n}}&=&\frac{\mathcal{S}_{n+1}z_{0}+\mathcal{S}_{n+2}z_{
-1}+\mathcal{S}_{n}qz_{-2}}{\mathcal{S}_{n}z_{0}+\mathcal{S}_{n+1}z_{
-1}+\mathcal{S}_{n-1}qz_{-2}.}\\
&=&\frac{\mathcal{S}_{n+1}\frac{z_{0}}{z_{-1}}+\mathcal{S}_{n+2}+\mathcal{S}_{n-2}q\frac{z_{-2}}{z_{-1}}}{\mathcal{S}_{n}\frac{z_{0}}{z_{-1}}+\mathcal{S}_{n+1}
+\mathcal{S}_{n-1}q\frac{z_{-2}}{z_{-1}}}\\
&=&\frac{\mathcal{S}_{n+1}x_{0}+\mathcal{S}_{n+2}+\mathcal{S}_{n}q\frac{1}{x_{-1}}}{\mathcal{S}_{n}x_{0}+\mathcal{S}_{n+1}
+\mathcal{S}_{n-1}q\frac{1}{x_{-1}}}\\
&=&\frac{\mathcal{S}_{n+1}x_{0}x_{-1}+\mathcal{S}_{n+2}x_{-1}+\mathcal{S}_{n}q}{\mathcal{S}_{n}x_{0}x_{-1}+\mathcal{S}_{n+1}x_{-1}
+\mathcal{S}_{n-1}q}.
\end{eqnarray*}
The above computations prove the following result.
\begin{theorem}
Let $\left\{ x_{n}\right\} _{n\geq -1}$ be a solution of
\eqref{eq1}. Then, for $n=1,2,\ldots,$
\begin{equation}
x_{n}=\frac{\mathcal{S}_{n+1}x_{-1}+\mathcal{S}_{n}x_{0}x_{-1}+q\mathcal{S}_{n-1}}
{\mathcal{S}_{n}x_{-1}+\mathcal{S}_{n-1}x_{0}x_{-1}+q\mathcal{S}_{n-2}}.\label{eq11}
\end{equation}
where   the initial conditions $x_{-1},x_{0}\in \mathbb{R}-F$, with
$F$ is the Forbidden Set of equation \eqref{eq1} given by
\[F=\bigcup_{n=-1}^{\infty}\bigg\{(x_{-1},x_{0}):\mathcal{S}_{n}x_{-1}+\mathcal{S}_{n-1}x_{0}x_{-1}+q\mathcal{S}_{n-2}=0\bigg\}.\]\label{th1}
\end{theorem}
If $\alpha=\beta=\gamma$, then from \eqref{eq11} we get
\begin{equation*}
x_{n}=\frac{\mathcal{P}_{n+1}x_{-1}+\mathcal{P}_{n}x_{0}x_{-1}+q\mathcal{P}_{n-1}}
{\mathcal{P}_{n}x_{-1}+\mathcal{P}_{n-1}x_{0}x_{-1}+q\mathcal{P}_{n-2}}.\label{eq12}
\end{equation*}
Hence, for $\alpha=\beta=\gamma$ we have $\mathcal{S}_{n}=\mathcal{P}_{n}, n\in \mathbb{N}$, and consequently
we get the solution given in \cite{yazlik}.

\subsection{Global stability of solutions of equation \eqref{eq0}}
In this section we study the global stability character of the
solutions of equation \eqref{eq1}.
It is easy to show that eqref{eq1} has a unique positive
equilibrium point given by $\overline{x}=\phi$.
Let $I=(0,+\infty)$, and consider the function $f:\;I^{2}\longrightarrow I$ defined by
\[f(x,y)=\frac{py+q}{xy}.\]

\begin{theorem}
The equilibrium point $\overline{x}$ is locally asymptotically
stable.\label{th3}
\end{theorem}
\begin{proof}
The linearized equation of equation \eqref{eq1} about the
equilibrium $\overline{x}$ is
\[y_{n+1}=t_{1}y_{n}+t_{2}y_{n-1}\]
where
\[t_{1}=\frac{\partial f}{\partial x}(\overline{x},\overline{x})=-\frac{pR^{2}+12p^{2}+6qR}{R^{6}+pR^{2}+12p^{2}+\frac{48p^{3}}{R^{2}}}\]
and
\[t_{2}=\frac{\partial f}{\partial y}(\overline{x},\overline{x})=-\frac{6qR}{R^{6}+pR^{2}+12p^{2}+\frac{48p^{3}}{R^{2}}}\]
 and the
characteristic polynomial is
\[\lambda^{2}+t_{1}\lambda+t_{2}=0.\]
Consider the two functions defined by \[a(\lambda)=\lambda^{2},\quad
b(\lambda)=-(t_{1}\lambda+t_{2}).\] We have
\[\left|\frac{pR^{2}+12p^{2}+12qR}{R^{6}+pR^{2}+12p^{2}+\frac{48p^{3}}{R^{2}}}\right|<1.\]
Then
\begin{equation*}
\left\vert b(\lambda )\right\vert < \left\vert a(\lambda
)\right\vert, \quad \forall \lambda :\left\vert \lambda \right\vert = 1
\end{equation*}
Thus, by Rouche's theorem, all zeros of $P(\lambda
)=a(\lambda)-b(\lambda)=0$ lie in $|\lambda |<1$. So, by Theorem
\eqref{th6} we get that $\overline{x}$ is locally asymptotically stable.

\end{proof}

\begin{theorem}
The equilibrium point $\overline{x}$ is globally asymptotically stable.
\end{theorem}

\begin{proof}
Let $\left\{x_{n}\right\} _{n\geq -k}$ be a solution of equation
\eqref{eq1}. By Theorem \eqref{th3} we need only to prove that $E$
is global attractor, that is
$$\displaystyle\lim_{n\rightarrow\infty}x_{n}=\phi.$$
 it follows from Theorem \eqref{th1} that

\begin{eqnarray*}
\lim_{n\rightarrow\infty}x_{n}&=&\lim_{n\rightarrow\infty}\frac{\mathcal{S}_{n+1}x_{-1}+
\mathcal{S}_{n}x_{0}x_{-1}+q\mathcal{S}_{n-1}}{\mathcal{S}_{n}x_{-1}+\mathcal{S}_{n-1}x_{0}x_{-1}+q\mathcal{S}_{n-2}}\\
&=&\lim_{n\rightarrow\infty}\frac{\mathcal{S}_{n}\left(\frac{\mathcal{S}_{n+1}}{\mathcal{S}_{n}}x_{-1}+
x_{0}x_{-1}+q\frac{\mathcal{S}_{n-1}}{\mathcal{S}_{n}}\right)}{\mathcal{S}_{n}
\left(x_{-1}+\frac{\mathcal{S}_{n-1}}{\mathcal{S}_{n}}x_{0}x_{-1}+q\frac{\mathcal{S}_{n-2}}{\mathcal{S}_{n}}\right)}\\
&=&\lim_{n\rightarrow\infty}\frac{\frac{\mathcal{S}_{n+1}}{\mathcal{S}_{n}}x_{-1}+
x_{0}x_{-1}+q\frac{\mathcal{S}_{n-1}}{\mathcal{S}_{n}}}{
x_{-1}+\frac{\mathcal{S}_{n-1}}{\mathcal{S}_{n}}x_{0}x_{-1}+q\frac{\left(\frac{1}{q}\mathcal{S}_{n+1}-\frac{p}{q}\mathcal{S}_{n-1}\right)}{\mathcal{S}_{n}}}\\
&=&\lim_{n\rightarrow\infty}\frac{\phi\left( x_{-1}+
\frac{1}{\phi}x_{0}x_{-1}+q\frac{1}{\phi}\right)}{
x_{-1}+\frac{1}{\phi}x_{0}x_{-1}+\phi+\frac{p}{\phi}}
\end{eqnarray*}
Then
$$\displaystyle\lim_{n\rightarrow\infty}x_{n}=\phi.$$

\end{proof}

\begin{example} For confirming results of this section, we consider the following numerical example.
Let $\alpha=2, \beta= 5$ and $\gamma=4$ in \eqref{eq0}, then we obtain the equation
\begin{equation}
x_{n+1}=\frac{2x_{n-1}+5}{ 4 x_{n}x_{n-1}}.\label{eq80}
\end{equation}
Assume $x_{-1}=3$ and  $x_{0}=0.2$, (see Fig.
\ref{iresp1}).\label{ex1}
\begin{figure}[ht]
\centerline{\epsfig{figure=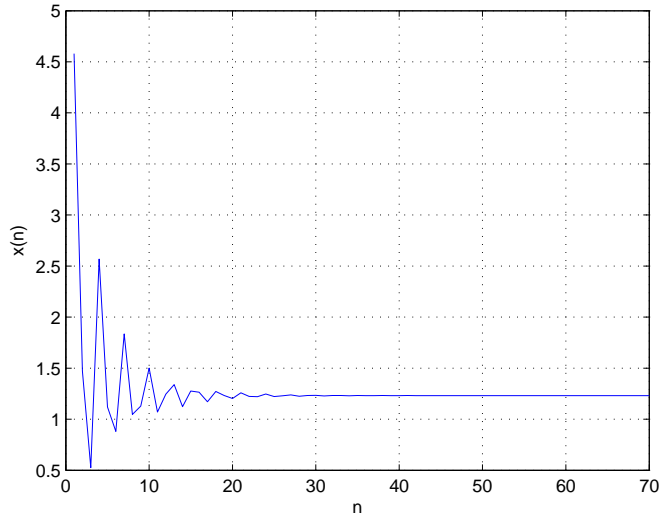, width=100mm}} \caption{This
figure shows that the solution of the equation \eqref{eq80} is global
attractor, that is,
$\displaystyle\lim_{n\rightarrow\infty}x_{n}=\phi.$}
\label{iresp1}
\end{figure}
\end{example}

\section{Closed-form and stability of solutions of system \eqref{eq01}}\label{sec4}

\subsection{Closed-form solutions of system \eqref{eq01}}
In this section, we derive the respective solution form of system \eqref{eq01}. We put  $q=\dfrac{\alpha}{\gamma}$ and $p=\dfrac{\beta}{\gamma}$.
Hence, we have the system
\begin{equation}
	x_{n+1} = \frac{p x_{n-1}  + q}{ y_n x_{n-1}}, \qquad y_{n+1} = \frac{p y_{n-1} +q}{ x_n y_{n-1}} ,\qquad n\in \mathbb{N}_0\label{eq010}
\end{equation}

The following theorem describes the form of the solutions of system \eqref{eq010}.
\begin{theorem}
Let $\left\{ x_{n},y_{n}\right\} _{n\geq -1}$ be a solution of
\eqref{eq010}. Then for $n=1,2,\ldots,$

\begin{equation}
x_{n}=\left\{
\begin{array}{ll}
\dfrac{\mathcal{S}_{n+1}y_{-1}+\mathcal{S}_{n}x_{0}y_{-1}+q\mathcal{S}_{n-1}}
{\mathcal{S}_{n}y_{-1}+\mathcal{S}_{n-1}x_{0}y_{-1}+q\mathcal{S}_{n-2}},& if\ n\ is\ even,\\[1em]
\dfrac{\mathcal{S}_{n+1}x_{-1}+\mathcal{S}_{n}y_{0}x_{-1}+q\mathcal{S}_{n-1}}
{\mathcal{S}_{n}x_{-1}+\mathcal{S}_{n-1}y_{0}x_{-1}+q\mathcal{S}_{n-2}},& if\ n\ is\ odd,
\end{array}
\right.\\
\end{equation}
\begin{equation}
y_{n}=\left\{
\begin{array}{ll}
\dfrac{\mathcal{S}_{n+1}x_{-1}+\mathcal{S}_{n}y_{0}x_{-1}+q\mathcal{S}_{n-1}}
{\mathcal{S}_{n}x_{-1}+\mathcal{S}_{n-1}y_{0}x_{-1}+q\mathcal{S}_{n-2}},& if\ n\ is\ even,\\[1em]
\dfrac{\mathcal{S}_{n+1}y_{-1}+\mathcal{S}_{n}x_{0}y_{-1}+q\mathcal{S}_{n-1}}
{\mathcal{S}_{n}y_{-1}+\mathcal{S}_{n-1}x_{0}y_{-1}+q\mathcal{S}_{n-2}},& if\ n\ is\ odd,
\end{array}
\right.
\end{equation}
where   the initial conditions $x_{-1}$, $x_{0}$, $y_{-1}$ and $y_{0}\in \mathbb{R}\setminus\left(F_{1}\cup F_{2}\right)$, with
$F_{1}$ and $F_{2}$ are the forbidden sets of equation \eqref{eq1} given by
\[F_{1}=\bigcup_{n=-1}^{\infty}\bigg\{(x_{-1},x_{0},y_{-1},y_{0}):\mathcal{S}_{n}x_{-1}+\mathcal{S}_{n-1}y_{0}x_{-1}+q\mathcal{S}_{n-2}=0\bigg\},\]
and
\[F_{2}=\bigcup_{n=-1}^{\infty}\bigg\{(x_{-1},x_{0},y_{-1},y_{0}):\mathcal{S}_{n}y_{-1}+\mathcal{S}_{n-1}x_{0}y_{-1}+q\mathcal{S}_{n-2}=0\bigg\}.\]
\label{th12}
\end{theorem}
\begin{proof}
The closed-form solution of \eqref{eq010} can be established through a similar approach we used in proving the one-dimensional case.
However, for convenience, we shall prove the theorem by induction.
For the basis step, we have
\[x_{1}=\frac{px_{-1}+q}{y_{0}x_{-1}} \quad \text{and}\quad  y_{1}=\frac{py_{-1}+q}{x_{0}y_{-1}},\]
so the result clearly holds for $n=0$.  Suppose that $n>0$ and that our
assumption holds for $n-1$. That is,

\begin{align*}
x_{2n-2} &=\dfrac{\mathcal{S}_{2n-1}y_{-1}+\mathcal{S}_{2n-2}x_{0}y_{-1}+q\mathcal{S}_{2n-3}}
{\mathcal{S}_{2n-2}y_{-1}+\mathcal{S}_{2n-3}x_{0}y_{-1}+q\mathcal{S}_{2n-4}},\\
x_{2n-1} &=\dfrac{\mathcal{S}_{2n}x_{-1}+\mathcal{S}_{2n-1}y_{0}x_{-1}+q\mathcal{S}_{2n-2}}
{\mathcal{S}_{2n-1}x_{-1}+\mathcal{S}_{2n-2}y_{0}x_{-1}+q\mathcal{S}_{2n-3}},\\
y_{2n-2} &=\dfrac{\mathcal{S}_{2n-1}x_{-1}+\mathcal{S}_{2n-2}y_{0}c_{-1}+q\mathcal{S}_{2n-3}}
{\mathcal{S}_{2n-2}x_{-1}+\mathcal{S}_{2n-3}y_{0}x_{-1}+q\mathcal{S}_{2n-4}},\\
y_{2n-1} &=\dfrac{\mathcal{S}_{2n}y_{-1}+\mathcal{S}_{2n-1}x_{0}y_{-1}+q\mathcal{S}_{2n-2}}
{\mathcal{S}_{2n-1}y_{-1}+\mathcal{S}_{2n-2}x_{0}y_{-1}+q\mathcal{S}_{2n-3}}.\\
\end{align*}
Now it follows from  system \eqref{eq010} that%
\begin{align*}
x_{2n} &=\dfrac{px_{2n-2}+q}{y_{2n-1}x_{2n-2}}\\
&=\dfrac{p\dfrac{\mathcal{S}_{2n-1}y_{-1}+\mathcal{S}_{2n-2}x_{0}y_{-1}+q\mathcal{S}_{2n-3}}
{\mathcal{S}_{2n-2}y_{-1}+\mathcal{S}_{2n-3}x_{0}y_{-1}+q\mathcal{S}_{2n-4}}+q}{\dfrac{\mathcal{S}_{2n}y_{-1}+\mathcal{S}_{2n-1}x_{0}y_{-1}+q\mathcal{S}_{2n-2}}
{\mathcal{S}_{2n-1}y_{-1}+\mathcal{S}_{2n-2}x_{0}y_{-1}+q\mathcal{S}_{2n-3}}\dfrac{\mathcal{S}_{2n-1}y_{-1}+\mathcal{S}_{2n-2}x_{0}y_{-1}+q\mathcal{S}_{2n-3}}
{\mathcal{S}_{2n-2}y_{-1}+\mathcal{S}_{2n-3}x_{0}y_{-1}+q\mathcal{S}_{2n-4}}}\\
&=\dfrac{p(\mathcal{S}_{2n-1}y_{-1}+\mathcal{S}_{2n-2}x_{0}y_{-1}+q\mathcal{S}_{2n-3})+q
(\mathcal{S}_{2n-2}y_{-1}+\mathcal{S}_{2n-3}x_{0}y_{-1}+q\mathcal{S}_{2n-4})}{\mathcal{S}_{2n}y_{-1}+\mathcal{S}_{2n-1}x_{0}y_{-1}+q\mathcal{S}_{2n-2}}
\end{align*}
So, we have
\[x_{2n}=\dfrac{\mathcal{S}_{2n+1}y_{-1}+\mathcal{S}_{2n}x_{0}y_{-1}+q\mathcal{S}_{2n-1}}
{\mathcal{S}_{2n}y_{-1}+\mathcal{S}_{2n-1}x_{0}y_{-1}+q\mathcal{S}_{2n-2}}.\]
Also it follows from system \eqref{eq010} that
\begin{align*}
y_{2n} &=\dfrac{py_{2n-2}+q}{x_{2n-1}y_{2n-2}}\\
&=\dfrac{p\dfrac{\mathcal{S}_{2n-1}x_{-1}+\mathcal{S}_{2n-2}y_{0}x_{-1}+q\mathcal{S}_{2n-3}}
{\mathcal{S}_{2n-2}x_{-1}+\mathcal{S}_{2n-3}y_{0}x_{-1}+q\mathcal{S}_{2n-4}}+q}
{\dfrac{\mathcal{S}_{2n}x_{-1}+\mathcal{S}_{2n-1}y_{0}x_{-1}+q\mathcal{S}_{2n-2}}
{\mathcal{S}_{2n-1}x_{-1}+\mathcal{S}_{2n-2}y_{0}x_{-1}+q\mathcal{S}_{2n-3}}
\dfrac{\mathcal{S}_{2n-1}x_{-1}+\mathcal{S}_{2n-2}y_{0}c_{-1}+q\mathcal{S}_{2n-3}}
{\mathcal{S}_{2n-2}x_{-1}+\mathcal{S}_{2n-3}y_{0}x_{-1}+q\mathcal{S}_{2n-4}}}\\
&=\dfrac{p(\mathcal{S}_{2n-1}x_{-1}+\mathcal{S}_{2n-2}y_{0}x_{-1}+q\mathcal{S}_{2n-3})+
q(\mathcal{S}_{2n-2}x_{-1}+\mathcal{S}_{2n-3}y_{0}x_{-1}+q\mathcal{S}_{2n-4})}
{\mathcal{S}_{2n}x_{-1}+\mathcal{S}_{2n-1}y_{0}x_{-1}+q\mathcal{S}_{2n-2}}.
\end{align*}
Hence, we have
\[y_{2n}=\dfrac{\mathcal{S}_{2n+1}x_{-1}+\mathcal{S}_{2n}y_{0}c_{-1}+q\mathcal{S}_{2n-1}}
{\mathcal{S}_{2n}x_{-1}+\mathcal{S}_{2n-1}y_{0}x_{-1}+q\mathcal{S}_{2n-2}}.\]
Using the same argument it follows from system \eqref{eq010} that
\begin{align*}
x_{2n+1} &=\dfrac{px_{2n-1}+q}{y_{2n}x_{2n-1}}\\
&=\dfrac{p\dfrac{\mathcal{S}_{2n}x_{-1}+\mathcal{S}_{2n-1}y_{0}x_{-1}+q\mathcal{S}_{2n-2}}
{\mathcal{S}_{2n-1}x_{-1}+\mathcal{S}_{2n-2}y_{0}x_{-1}+q\mathcal{S}_{2n-3}}+q}
{\dfrac{\mathcal{S}_{2n+1}x_{-1}+\mathcal{S}_{2n}y_{0}x_{-1}+q\mathcal{S}_{2n-1}}
{\mathcal{S}_{2n}x_{-1}+\mathcal{S}_{2n-1}y_{0}x_{-1}+q\mathcal{S}_{2n-2}}
\dfrac{\mathcal{S}_{2n}x_{-1}+\mathcal{S}_{2n-1}y_{0}x_{-1}+q\mathcal{S}_{2n-2}}
{\mathcal{S}_{2n-1}x_{-1}+\mathcal{S}_{2n-2}y_{0}x_{-1}+q\mathcal{S}_{2n-3}}}\\
&=\dfrac{p(\mathcal{S}_{2n}x_{-1}+\mathcal{S}_{2n-1}y_{0}x_{-1}+q\mathcal{S}_{2n-2})+
q(\mathcal{S}_{2n-1}x_{-1}+\mathcal{S}_{2n-2}y_{0}x_{-1}+q\mathcal{S}_{2n-3})}
{\mathcal{S}_{2n+1}x_{-1}+\mathcal{S}_{2n}y_{0}x_{-1}+q\mathcal{S}_{2n-1}}.
\end{align*}%
This  yields
\[x_{2n+1}=\dfrac{\mathcal{S}_{2n+2}x_{-1}+\mathcal{S}_{2n+1}y_{0}x_{-1}+q\mathcal{S}_{2n}}
{\mathcal{S}_{2n+1}x_{-1}+\mathcal{S}_{2n}y_{0}c_{-1}+q\mathcal{S}_{2n-1}}.\]
Moreover, we have
\begin{align*}
y_{2n+1} &=\dfrac{py_{2n-1}+q}{x_{2n}y_{2n-1}}\\
&=\dfrac{p\dfrac{\mathcal{S}_{2n}y_{-1}+\mathcal{S}_{2n-1}x_{0}y_{-1}+q\mathcal{S}_{2n-2}}
{\mathcal{S}_{2n-1}y_{-1}+\mathcal{S}_{2n-2}x_{0}y_{-1}+q\mathcal{S}_{2n-3}}+q}
{\dfrac{\mathcal{S}_{2n+1}y_{-1}+\mathcal{S}_{2n}x_{0}y_{-1}+q\mathcal{S}_{2n-1}}
{\mathcal{S}_{2n}y_{-1}+\mathcal{S}_{2n-1}x_{0}y_{-1}+q\mathcal{S}_{2n-2}}
\dfrac{\mathcal{S}_{2n}y_{-1}+\mathcal{S}_{2n-1}x_{0}y_{-1}+q\mathcal{S}_{2n-2}}
{\mathcal{S}_{2n-1}y_{-1}+\mathcal{S}_{2n-2}x_{0}y_{-1}+q\mathcal{S}_{2n-3}}}\\
&=\dfrac{p(\mathcal{S}_{2n}y_{-1}+\mathcal{S}_{2n-1}x_{0}y_{-1}+q\mathcal{S}_{2n-2})+
q(\mathcal{S}_{2n-1}y_{-1}+\mathcal{S}_{2n-2}x_{0}y_{-1}+q\mathcal{S}_{2n-3})}
{\mathcal{S}_{2n+1}y_{-1}+\mathcal{S}_{2n}x_{0}y_{-1}+q\mathcal{S}_{2n-1}},
\end{align*}%
and this implies that
\[y_{2n+1}=\dfrac{\mathcal{S}_{2n+2}y_{-1}+\mathcal{S}_{2n+1}x_{0}y_{-1}+q\mathcal{S}_{2n}}
{\mathcal{S}_{2n+1}y_{-1}+\mathcal{S}_{2n}x_{0}y_{-1}+q\mathcal{S}_{2n-1}}.\]
This completes the proof of the theorem.
\end{proof}

\subsection{Global attractor of  solutions of system \eqref{eq01}}
Our aim in this section is to study the asymptotic behavior of
positive solutions of system \eqref{eq010}. Let $I=J=(0,+\infty)$,
and consider the functions
\[f:\;I^{2}\times
J^{2}\longrightarrow I \quad \text{and}\quad  g:\;I^{2}\times J^{2} \longrightarrow J\] defined by
\[f(u_{0},u_{1},v_{0},v_{1})=\frac{pu_{1}+q}{v_{0}u_{1}} \quad \text{and}\quad
g(u_{0},u_{1},v_{0},v_{1})=\frac{pv_{1}+q}{u_{0}v_{1}},\]
respectively.
\begin{lemma} System \eqref{eq1} has a unique equilibrium point in $I\times
J$, namely
$$E=\left(\dfrac{R^{2}+12p}{6R},\dfrac{R^{2}+12p}{6R}\right).$$
\end{lemma}

\begin{proof} Clearly the system
$$\overline{x}=\frac{p\overline{x}+ q}{\overline{x}\overline{y}},\quad\overline{y}=\frac{p\overline{y}+q}{\overline{y}\overline{x}},$$
has a unique solution in $I^{2}\times J^{2}$ which is
$$E=\left(\dfrac{R^{2}+12p}{6R},\dfrac{R^{2}+12p}{6R}\right).$$

\end{proof}

\begin{theorem}
The equilibrium point $E$ is global attractor.
\end{theorem}

\begin{proof}
Let $\left\{x_{n},y_{n}\right\} _{n\geq 0}$ be a solution of system
\eqref{eq1}.
Let $n\rightarrow\infty$ in Theorem \ref{th12}. That is, we have
\begin{align*}
\lim_{n\rightarrow\infty}x_{2n}&=\lim_{n\rightarrow\infty}\frac{\mathcal{S}_{n+1}y_{-1}+
\mathcal{S}_{n}x_{0}y_{-1}+q\mathcal{S}_{n-1}}{\mathcal{S}_{n}y_{-1}+\mathcal{S}_{n-1}x_{0}y_{-1}+q\mathcal{S}_{n-2}}\\
&=\lim_{n\rightarrow\infty}\frac{\mathcal{S}_{n}\left(\frac{\mathcal{S}_{n+1}}{\mathcal{S}_{n}}y_{-1}+
x_{0}y_{-1}+q\frac{\mathcal{S}_{n-1}}{\mathcal{S}_{n}}\right)}{\mathcal{S}_{n}
\left(y_{-1}+\frac{\mathcal{S}_{n-1}}{\mathcal{S}_{n}}x_{0}y_{-1}+q\frac{\mathcal{S}_{n-2}}{\mathcal{S}_{n}}\right)}\\
&=\lim_{n\rightarrow\infty}\frac{\frac{\mathcal{S}_{n+1}}{\mathcal{S}_{n}}y_{-1}+
x_{0}y_{-1}+q\frac{\mathcal{S}_{n-1}}{\mathcal{S}_{n}}}{
y_{-1}+\frac{\mathcal{S}_{n-1}}{\mathcal{S}_{n}}x_{0}y_{-1}+q\frac{\left(\frac{1}{q}\mathcal{S}_{n+1}-\frac{p}{q}\mathcal{S}_{n-1}\right)}{\mathcal{S}_{n}}}\\
&=\frac{\phi\left( y_{-1}+
\frac{1}{\phi}x_{0}y_{-1}+q\frac{1}{\phi}\right)}{
y_{-1}+\frac{1}{\phi}x_{0}y_{-1}+\phi+\frac{p}{\phi}}=\phi.
\end{align*}
and
\begin{align*}
\lim_{n\rightarrow\infty}x_{2n+1}&=\lim_{n\rightarrow\infty}\frac{\mathcal{S}_{n+1}x_{-1}+
\mathcal{S}_{n}y_{0}x_{-1}+q\mathcal{S}_{n-1}}{\mathcal{S}_{n}x_{-1}+\mathcal{S}_{n-1}y_{0}x_{-1}+q\mathcal{S}_{n-2}}\\
&=\lim_{n\rightarrow\infty}\frac{\mathcal{S}_{n}\left(\frac{\mathcal{S}_{n+1}}{\mathcal{S}_{n}}x_{-1}+
y_{0}x_{-1}+q\frac{\mathcal{S}_{n-1}}{\mathcal{S}_{n}}\right)}{\mathcal{S}_{n}
\left(x_{-1}+\frac{\mathcal{S}_{n-1}}{\mathcal{S}_{n}}y_{0}x_{-1}+q\frac{\mathcal{S}_{n-2}}{\mathcal{S}_{n}}\right)}\\
&=\lim_{n\rightarrow\infty}\frac{\frac{\mathcal{S}_{n+1}}{\mathcal{S}_{n}}x_{-1}+
y_{0}x_{-1}+q\frac{\mathcal{S}_{n-1}}{\mathcal{S}_{n}}}{
x_{-1}+\frac{\mathcal{S}_{n-1}}{\mathcal{S}_{n}}y_{0}x_{-1}+q\frac{\left(\frac{1}{q}\mathcal{S}_{n+1}-\frac{p}{q}\mathcal{S}_{n-1}\right)}{\mathcal{S}_{n}}}\\
&=\frac{\phi\left( x_{-1}+
\frac{1}{\phi}y_{0}x_{-1}+q\frac{1}{\phi}\right)}{
x_{-1}+\frac{1}{\phi}y_{0}x_{-1}+\phi+\frac{p}{\phi}}=\phi.
\end{align*}
Then
$\displaystyle\lim_{n\rightarrow\infty}x_{n}=\phi$.
Similarly, we obtain
$\displaystyle\lim_{n\rightarrow\infty}y_{n}=\phi$.
Thus, we have
\[\displaystyle\lim_{n\rightarrow\infty}(x_{n},y_{n})=E.\]
\end{proof}

\begin{example} As an illustration of our results, we consider the following numerical
example. Let $\alpha=2, \beta=3$ and $\gamma=5$ in system \eqref{eq01}, then we obtain the
system
\begin{equation}
	x_{n+1} = \frac{2x_{n-1}  + 3}{5 y_n x_{n-1}}, \qquad y_{n+1} = \frac{2 y_{n-1} +3}{ 5x_n y_{n-1}} ,\qquad n\in \mathbb{N}_0\label{eq011}
\end{equation}
Assume $x_{-1}=1.2, x_{0}=3.6, y_{-1}=2.3$  and $y_{0}=0.8.$ \emph{(See Fig.
\ref{iresp2})}.\label{ex4}
\begin{figure}[ht]
\centerline{\epsfig{figure=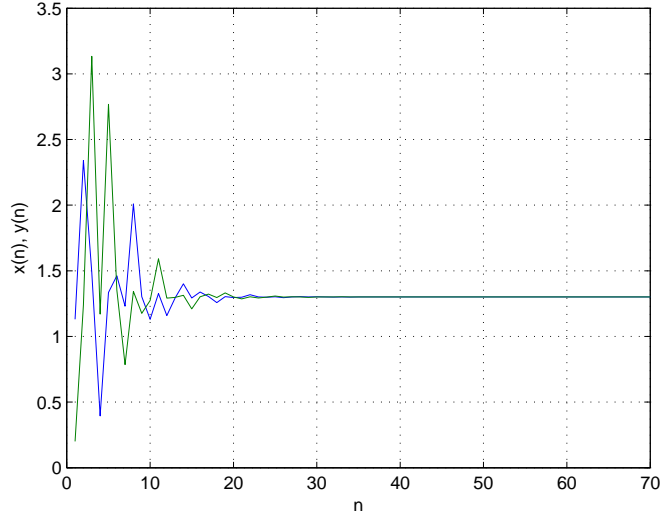, width=100mm}} \caption{This
figure shows that the solution of the system \eqref{eq011} is global
attractor, that is
$\displaystyle\lim_{n\rightarrow\infty}x_{n}=E.$}
\label{iresp2}
\end{figure}
\end{example}

\section{Summary and Recommendations}\label{sec5}

In this work, we have successfully established the closed-form solution of the rational difference equation
\begin{equation*}
x_{n+1}=\frac{\alpha x_{n-1}+\beta}{ \gamma x_{n}x_{n-1}}
\end{equation*}
as well as the closed-form solutions of its corresponding two-dimensional case
\begin{equation*}
	x_{n+1} = \frac{\alpha x_{n-1}  + \beta}{\gamma y_n x_{n-1}}, \qquad y_{n+1} = \frac{\alpha y_{n-1} +\beta}{\gamma x_n y_{n-1}}.
\end{equation*}
Also, we obtained stability results for the positive solutions of these systems.
Particularly, we have shown that the positive solutions of each of these equations tends to a computable finite number, and is in fact expressible in terms of the well-known plastic number.
Meanwhile, for future investigation, one could also derive the closed-form solution and examine the stability of solutions of the system
\begin{equation*}
	x_{n+1} = \frac{\alpha x_{n-1}  - \beta}{\gamma y_n x_{n-1}}, \qquad y_{n+1} = \frac{\alpha y_{n-1}  \pm \beta}{\gamma x_n y_{n-1}}.
\end{equation*}
This work we leave to the interested readers.

\end{document}